\documentclass[draft,reqno]{amsart}
\usepackage[american]{babel}
\usepackage{amssymb,euscript}
\numberwithin{equation}{section}
\newtheorem{theorem}{Theorem}[section]
\newtheorem{lemma}[theorem]{Lemma}
\newtheorem{corollary}[theorem]{Corollary}

\theoremstyle{definition}
\newtheorem{definition}[theorem]{Definition}
\theoremstyle{remark}
\newtheorem{remark}[theorem]{Remark}

\def\Re{\operatorname{{Re}}}
\def\Im{\operatorname{{Im}}}
\DeclareMathOperator*{\esssup}{ess\,sup}

\author[Faminskii]{Andrei V. Faminskii}
\address{RUDN University, 6 Miklukho-Maklaya Street, Moscow, 117198, Russian Federation}
\email{afaminskii@sci.pfu.edu.ru}
\thanks{This paper has been supported by Russian Science Foundation grant 23-21-00101.}

\title[Higher order nonlinear Schr\"{o}dinger equation]{On boundary controllability for the higher order nonlinear Schr\"{o}dinger equation}

\date{}
\begin{document}
\maketitle

\begin{abstract}
A control problem with final overdetermination is considered for the higher order nonlinear Schr\"{o}dinger equation on a bounded interval. The boundary condition on the space derivative is chosen as the control. Results on global existence of solutions under small input date are established.
\end{abstract}

\section{Introduction}\label{S1}

In this paper the higher order nonlinear Schr\"odinger equation (HNLS)
\begin{equation}\label{1}
i u_t  + a u_{xx} + i b u_x + i u_{xxx} + \lambda |u|^{p_0} u +
 i \beta \bigl( |u|^{p_1} u\bigr)_x + i \gamma \bigl( |u|^{p_1}\bigr)_x u= f(t,x),
\end{equation}
posed on an interval $I = (0,R)$, is considered. Here $a$, $b$, $\lambda$, $\beta$, $\gamma$ are real constants, $p_0, p_1\geq 1$, $u=u(t,x)$ and $f$ are complex-valued functions (as well as all other functions below, unless otherwise stated).  

For an arbitrary $T>0$ in a rectangle $Q_T = (0,T)\times I$ consider an initial-boundary value problem for equation \eqref{1} with an initial condition
\begin{equation}\label{2}
u(0,x) = u_0 (x),\quad x\in [0,R],
\end{equation}
and boundary conditions
\begin{equation}\label{3}
u(t,0) = \mu(t), \quad u(t,R) = \nu(t), \quad u_x(t,R) = h(t) ,\quad t\in [0,T],
\end{equation}
where the function $h$ is unknown and must be chosen such, that the corresponding solution of problem \eqref{1}--\eqref{3} satisfies the condition of terminal overdetermination
\begin{equation}\label{4}
u(T,x) = u_T (x),\quad x\in [0,R],
\end{equation}
for given function $u_T$.

Equation \eqref{1} is a generalized combination of the nonlinear Schr\"{o}dinger equation (NLS)
$$
i u_t + a u_{xx} + \lambda |u|^p u = 0
$$
and the Korteweg--de~Vries equation (KdV)
$$
u_t  + b u_x + u_{xxx} + u u_x =0.
$$
It has various physical applications, in particular, it models propagation of femtosecond optical pulses in a monomode optical fiber, accounting for additional effects such as third order dispersion, self-steeping of the pulse, and self-frequency shift (see \cite{Fib, HK, Kod, KC} and the references therein).

The first result on boundary controllability for the KdV equation on a boundary interval appeared in the pioneer paper by L.~Rosier \cite{Rosier}. In the case $b=1$, initial condition \eqref{2} and boundary conditions \eqref{3} for $\mu = \nu \equiv 0$ it was proved that under small $u_0, u_T \in L_2(0,R)$ there existed a solution under the restriction on the length of the interval
$$
R \ne 2\pi \sqrt{\frac{k^2 + kl +l^2}3}, \quad \forall k, l \in \mathbb N.
$$ 
In paper \cite{CPV} this result was extended to the truncated HNLS equation with cubic nonlinearity
$$
i u_t  + a u_{xx} + i b u_x + i u_{xxx} + |u|^{2} u = 0
$$
again under homogeneous boundary conditions \eqref{3}, under restriction on the length of the interval
\begin{equation}\label{5}
R \ne 2\pi \sqrt{\frac{k^2 + kl +l^2}{3b + a^2}}, \quad  \forall k, l \in \mathbb N,
\end{equation} 
and under the conditions $b>0$, $|a|<3$ (in fact, the equation considered in \cite{CPV} there was a positive coefficient before the third derivative, but it could be easily eliminated by the scaling with respect to $t$, which is possible since the time interval was arbitrary). The argument repeated the one from \cite{Rosier}.

In the present paper the same result is established for the general HNLS equation \eqref{1} under non-homogeneous boundary conditions \eqref{3} and without any conditions on the coefficients $a$ and $b$.

Note that in the recent paper \cite{FM} the inverse initial-boundary value problem \eqref{1}--\eqref{3} was considered with an integral overdetermination
$$
\int_0^R u(t,x) \omega(x) \,dx = \varphi(t), \quad t\in [0,T],
$$
for given functions $\omega$ and $\varphi$. Either boundary function $h$ or the function $F$ in the right-hand side $f(t,x) = F(t)g(t,x)$ for given function $g$ were chosen as controls. Results on well-posedness under either small input data or small time interval were established.

In \cite{CCFSV} a direct  initial-boundary value problem on a bounded interval with homogeneous boundary conditions \eqref{3} for equation \eqref{1} in the case $p_0=p_1=p$ was studied. For $p\in [1,2]$ and the initial function $u_0 \in H^s(I)$, $0\leq s \leq 3$, results on global existence and uniqueness of mild solutions were obtained. For $u_0\in L_2(I)$ the result on global existence was extended either to $p\in (2,3)$ or $p\in (2,4)$, $\gamma=0$. Non-homogeneous boundary conditions were considered in \cite{FL} in the real case and nonlinearity $uu_x$. Note also that in \cite{FM} there is a brief survey of other results concerning the direct initial-boundary value problems for equation \eqref{1}.

\bigskip

Solutions of the considered problems are constructed in a special functional space
$$
X(Q_T) = C([0,T]; L_2(I)) \cap L_2(0,T; H^1(I)),
$$
endowed with the norm
$$
\|u\|_{X(Q_T)} = \sup\limits_{t\in (0,T)} \|u(t,\cdot)\|_{L_2(I)} + \|u_x\|_{L_2(Q_T)}.
$$
For $r>0$ denote by $\overline X_r(Q_T)$ the closed ball $\{u\in X(Q_T): \|u\|_{X(Q_T)} \leq r\}$.

Introduce the notion of a weak solution of problem \eqref{1}--\eqref{3}

\begin{definition}\label{D1}
Let $u_0 \in L_2(I)$, $\mu, \nu, h \in L_2(0,T)$, $f\in L_1(Q_T)$. A function $u\in X(Q_T)$ is called a weak solution of problem \eqref{1}--\eqref{3} if $u(t,0) \equiv \mu(t)$, $u(t,R) \equiv \nu(t)$, and for all test functions $\phi(t,x)$, such that $\phi\in C^1([0,T]; L_2(I)) \cap C([0,T];(H^{3}\cap H^1_0)(I))$,  $\phi\big|_{t=T} \equiv 0$, $\phi_x\big|_{x=0} \equiv 0$, the functions
$|u|^{p_0} u, |u|^{p_1} u, |u|^{p_1} u_x \in L_1(Q_T)$, and the following integral identity is verified:
\begin{multline}\label{6}
\iint_{Q_T} \Bigl[ i u\phi_t  + a u_x \phi_x + i b u \phi_x - i u_x \phi_{xx} -
\lambda |u|^{p_0} u \phi + i \beta |u|^{p_1} u \phi_x + i \gamma |u|^{p_1} (u\phi)_x  \\ +f\phi \Bigr]\, dxdt  +
i \int_0^R u_0 \phi\big|_{t=0} \,dx + i \int_0^T h \phi_x\big|_{x=R} \,dt =0.
\end{multline}
\end{definition}

\begin{remark}\label{R1}
Note that $\phi, \phi_x \in C(\overline Q_T)$, $\phi_x \in C(\overline I; L_2(0,T))$, therefore, the integrals in \eqref{6} exist.
\end{remark}

To describe  the properties of the boundary data $\mu$ and $\nu$ introduce the fractional-order Sobolev spaces. Let 
$\widehat f(\xi)\equiv \mathcal F[f](\xi)$ and $\mathcal F^{-1}[f](\xi)$ be
the direct and inverse Fourier transforms of a function $f$ respectively. In particular, for
$f\in \mathcal S(\mathbb R)$
$$
\widehat f(\xi)=\int_{\mathbb R} e^{-i\xi x} f(x)\,dx,\qquad \mathcal F^{-1}[f](x)=\frac 1 {2\pi} \int_{\mathbb R} e^{i\xi x} f(\xi)\,d\xi3.
$$
For $s\in \mathbb R$ define the fractional-order Sobolev space
$$
H^s(\mathbb R)= \bigl\{f: \mathcal F^{-1}[(1+|\xi|^s)\widehat f(\xi)] \in L_2(\mathbb R)\bigr\}
$$
and for certain $T>0$ let $H^s(0,T)$ be a space of restrictions on $(0,T)$ of functions from $H^s(\mathbb R)$. 

Now we can pass to the main result of the paper.

\begin{theorem}\label{T1}
Let $p_0\in [1,4]$, $p_1\in [1,2]$, $u_0, u_T \in L_2(I)$, $\mu, \nu \in H^{1/3}(0,T)$, $f\in L_1(0,T;L_2(I))$. Assume also that if $3b + a^2 > 0$ condition \eqref{5} is satisfied.
Denote
\begin{equation}\label{7}
c_0 = \|u_0\|_{L_2(I)} + \|u_T\|_{L_2(I)} + \|\mu\|_{H^{1/3}(0,T)} + \|\nu\|_{H^{1/3}(0,T)}+ \|f\|_{L_1(0,T;L_2(I))}.
\end{equation}
Then there exists $\delta>0$ such that under the assumption $c_0\leq \delta$ there exists a function $h\in L_2(0,T)$ and the corresponding unique solution of problem \eqref{1}--\eqref{3} $u\in X(Q_T)$ verifying condition \eqref{4}.
\end{theorem}

\begin{remark}\label{R2}
The smoothness assumption $\mu,\nu \in H^{1/3}(0,T)$ on the boundary data is natural, since if one considers the initial value problem 
$$
v_t + v_{xxx} =0, \quad v\big|_{t=0} = v_0(x) \in L_2(\mathbb R),
$$
then, by \cite{KPV}, its solution $v\in C(\mathbb R; L_2(\mathbb R)$ (which can be constructed via the Fourier transform) satisfies the following relations for any $x\in \mathbb R$
$$
\|D^{1/3}_t v(\cdot,x)\|_{L_2(\mathbb R)} = \|v_x(\cdot,x)\|_{L_2(\mathbb R)} = c\|v_0\|_{L_2(\mathbb R)}.
$$
\end{remark}

\bigskip 
Further we use the following simple interpolating inequality: there exists a constant $c=c(R,q)$ such that for any $\varphi \in H^1(I)$ 
\begin{equation}\label{8}
\|\varphi\|_{L_\infty(I)} \leq c\|\varphi'\|_{L_2(I)}^{1/2} \|\varphi\|_{L_2(I)}^{1/2} +c\|\varphi\|_{L_2(I)},
\end{equation}
where the second term in the right-hand side is absent if $\varphi \in H_0^1(I)$.

The paper is organized as follows. In Section~\ref{S2} results on the corresponding linear problem are presented, Section~\ref{S3} contains the proof of nonlinear results.

\section{Auxiliary linear problem}\label{S2}

Besides the nonlinear problem consider its linear analogue and start with the following one with homogeneous boundary conditions
\begin{gather}\label{9}
i u_t  + a u_{xx} + i b u_x + i u_{xxx}  = f(t,x), \\
\label{10}
u\big|_{t=0} = u_0(x), \quad u\big|_{x=0} = u\big|_{x=R} = u_x\big|_{x=R} =0. 
\end{gather}
Define an operator
$$
A:D(A) \rightarrow L_2(I),\quad
y \mapsto A(y) =-y'''+i a y'' - b y'
$$
with the domain $D(A) = \{y\in H^3(I): y(0) = y(R) = y'(R)=0\}$.

\begin{lemma}\label{L1}
The operator $A$ generates a continuous semi-group of contractions $\left\{e^{tA}, t\geq 0 \right\}$ in $L_2(I)$.
\end{lemma}

\begin{proof}
This assertion is proved in \cite[Lemma~4.1]{CPV} but under the restriction $|a|<3$. However, the slight correction of that proof provides the desired result. In fact, the operator $A$ is obviously closed. Next, for $y\in D(A)$
$$
(Ay, y)_{L_2(I)} = \int_0^R (-y'''+i a y'' - b y')\bar y \,dx.
$$
Here,
\begin{gather*}
-\int_0^R y''' \bar y \, dx = - |y'(0)|^2 + \int_0^R y \bar y''' \,dx, \\
i\int_0^R y'' \bar y \,dx = - i\int_0^R |y'|^2 \,dx, \\
-\int_0^R y' \bar y \, dx = \int_0^R y \bar y' \,dx,
\end{gather*}
therefore,
$$
\Re (Ay,y)_{L_2(I)} = - \frac12 |y'(0)|^2 \leq 0
$$
and so the operator $A$ is dissipative. Next, the operator $A^* y = y''' - i a y'' +b y'$ with the domain $D(A^*) = \{y\in H^3(I): y(0) = y'(0) = y(R)=0\}$ and similarly for $y\in D(A^*)$
$$
\Re (A^*y,y)_{L_2(I)} = - \frac12 |y'(R)|^2 \leq 0.
$$
Therefore, the operator $A^*$ is also dissipative. Application of the Lumer--Phillips theorem  (see \cite{Pazy}) finishes the proof.
\end{proof}

\begin{remark}\label{R3}
Note that the weak solution of problem \eqref{9}, \eqref{10} can be considered in the space $L_1(0,T;L_2(I))$ in the sense of an integral identity
$$
\iint_{Q_T} u (i\phi_t  - a \phi_{xx} + i b \phi_x + i \phi_{xxx}) \, dxdt  + 
\iint_{Q_T} f \phi \,dxdt +
i \int_0^R u_0 \phi\big|_{t=0} \,dx =0,
$$
valid for any test function from Definition~\ref{D1}. Then the general theory of semi-groups (see \cite{Pazy}) provides that for $u_0 \in L_2(I)$, $f\in L_1(0,T;L_2(I))$ there exists a weak solution $u\in C([0,T]; L_2(I))$ of problem \eqref{9}, \eqref{10}, 
$$
u(t,\cdot) = e^{tA} u_0 + \int_0^t e^{(t - \tau) A} f(\tau,\cdot) \,d\tau,
$$
\begin{equation}\label{11}
\|u\|_{C([0,T]; L_2(I))} \leq \|u_0\|_{L_2(I)} + \|f\|_{L_1(0,T;L_2(I))},
\end{equation}
which is unique in $L_1(0,T;L_2(I))$. Moreover, for $u_0 \in D(A)$, $f\in C^1([0,T]; L_2(I))$ this solution is regular, that is, $u\in C^1([0,T];L_2(I)) \cap C([0,T];D(A))$. 
\end{remark}

\begin{lemma}\label{L2}
Let $u_0\in L_2(I)$, $f\equiv f_0 - f_{1x}$, where $f_0\in L_1(0,T;L_2(I))$, $f_1\in L_2(Q_T)$. Then there exist a unique weak solution to problem \eqref{9}, \eqref{10} $u\in X(Q_T)$ and a function $\theta\in L_2(0,T)$, such that for certain constant $c=c(T)$, non-decreasing with respect to $T$,
\begin{equation}\label{12}
\|u\|_{X(Q_T)} + \|\theta\|_{L_2(0,T)} \leq c\bigl(\|u_0\|_{L_2(I)} + 
\|f_0\|_{L_1(0,T;L_2(I))} +\|f_1\|_{L_2(Q_T)}\bigr),
\end{equation}
and for a.e. $t\in (0,T)$
\begin{multline}\label{13}
\frac{d}{dt}\int_0^R |u(t,x)|^2\rho(x)\,dx + |\theta(t)|^2 +
3\int_0^R |u_x|^2\rho' \,dx  \\
= b \int_0^R |u|^2\rho' \,dx   
+2a \Im \int_0^R u_x \bar{u}\rho' \,dx 
+ 2\Im \int_0^R f_0\bar{u}\rho \,dx 
+ 2 \Im \int_0^R f_1 (\bar{u}\rho)_x \,dx, 
\end{multline}
where either $\rho(x)\equiv 1$ or $\rho(x)\equiv 1+x$. Moreover, if $u_0 \in D(A)$ and $f\in C^1([0,T];L_2(I))$, then $\theta \equiv u_x\big|_{x=0}$.
\end{lemma}

\begin{proof}
First, consider regular solutions in the case $ u_0\in D(A)$, $f\in C^1([0,T];L_2(I))$. Then multiplying equality \eqref{9} by $2\bar{u}(t,x)\rho(x)$, extracting the imaginary part and integrating one obtains  an equality
\begin{multline}\label{14}
\int_0^R |u(t,x)|^2\rho(x)\,dx + \int_0^t |u_x(\tau,0)|^2\,d\tau +
3\iint_{Q_t} |u_x|^2\rho' \,dx d\tau \\
= \int_0^R |u_0|^2\rho\,dx + b\iint_{Q_t} |u|^2\rho' \,dx d\tau  
+2a \Im \iint_{Q_t} u_x \bar{u}\rho' \,dx d\tau \\
+ 2\Im \iint_{Q_t} f_0\bar{u}\rho \,dx d\tau
+ 2 \Im \iint_{Q_t} f_1 (\bar{u}\rho)_x \,dx d\tau.
\end{multline}
Choose $\rho(x)\equiv 1+x$, then
$$
\Bigl| 2a\int_0^R u_{x}\bar{u}\,dx \Bigr|\leq 
a^2 \int_0^R |u|^2\,dx + \int_0^R |u_{x}|^2\,dx,
$$
$$	
\Bigl| 2\int_0^R f_1 (\bar{u}\rho)_x \,dx \Bigr|\leq 
\bigl((1+R)^2+1\bigr)\int_0^R |f_1|^2\,dx+\int |u_{x}|^2\,dx + \int_0^R |u|^2\,dx,
$$
and equality \eqref{14} provides estimate \eqref{12} in the regular case.
This estimate gives an opportunity to establish existence of a weak solution with property \eqref{12} in the general case via closure. Moreover, equality \eqref{14} is also verified. In particular, this equality implies that the function 
$\|u(t,\cdot)\rho^{1/2}\|_{L_2(I)}^2$ is absolutely continuous on $[0,T]$ and then \eqref{13} follows.
\end{proof}

\begin{corollary}\label{C1}
There exists a linear bounded operator $P: L_2(I) \rightarrow L_2(0,T)$ such that for any $u_0\in L_2(I)$
\begin{equation}\label{15}
\|Pu_0\|_{L_2(0,T)} \leq \|u_0\|_{L_2(I)},
\end{equation}
for the corresponding weak solution $u\in X(Q_T)$ of problem \eqref{9}, \eqref{10} in the case $f_0=f_1 \equiv 0$,
\begin{equation}\label{16}
\|u_0\|^2_{L_2(I)} \leq  \frac1T \|u\|^2_{L_2(Q_T)} + 
\|Pu_0\|^2_{L_2(0,T)},
\end{equation}
and $Pu_0 = u_x\big|_{x=0}$ if $u_0 \in D(A)$.
\end{corollary}

\begin{proof}
In the case $|a| <3$ this assertion was proved in \cite[Lemma~4.2]{CPV}.  
Choosing in \eqref{13} $\rho(x) \equiv 1$ we obtain estimate \eqref{15} for $Pu_0 \equiv \theta$. Next, again for $\rho(x) \equiv 1$,
multiplying equality \eqref{13} by $(T-t)$ and integrating with respect to $t$, we derive an equality
$$
\iint_{Q_T} |u|^2 \,dxdt - T \int_0^R |u_0|^2 \,dx  + \int_0^T (T-t) 
|\theta(t)|^2 \,dt =0,
$$
which implies inequality \eqref{16}.
\end{proof}

Three following lemmas are proved in \cite{CPV} in the case $|a|<3$, $b>0$. The proof in the general case is similar, however, we present it here, moreover, in a more transparent way.
The first auxiliary lemma is concerned with the properties of the operator $A$.

\begin{lemma}\label{L3}
Let the function $y\in D(A)$, $y\not\equiv 0$, be the eigenfunction of the operator $A$ and $y'
(0) =0$. Then $3b + a^2 >0$ and $R = 2\pi \sqrt{(k^2 + kl + l^2)/(3b + a^2)}$ for certain natural numbers $k$ and $l$.
\end{lemma}

\begin{proof}
Let $\varkappa = y''(0)$, $\sigma = y''(R)$, $Ay = \lambda y$ for certain $\lambda \in \mathbb C$.

Extend the function $y$ by zero outside the segment $[0,R]$, note that $y\in H^2(\mathbb R)$. 
Then in $\EuScript S'(\mathbb R)$
$$
\lambda y + y''' -i a y'' + b y' = \varkappa \delta_0 - \sigma \delta_R, 
$$
where $\delta_{x_0}$ denotes the Dirac measure at the point $x_0$. Applying the Fourier transform we derive an equality
$$
(\lambda - i\xi^3 + i a \xi^2 + i b \xi) \widehat y(\xi) = \varkappa - \sigma e^{-i R \xi},
$$
whence for $p = i\lambda$
$$
\widehat y(\xi) = i\frac{\varkappa - \sigma e^{-i R \xi}}{\xi^3 - a \xi^2 - b \xi + p}.
$$
Since the function $y$ has the compact support, the function $\widehat y$ can be extended to the entire function on $\mathbb C$. Note that $(\varkappa,\sigma) \ne (0,0)$, otherwise $y\equiv 0$. The roots of the function $\varkappa - \sigma e^{-i R \xi}$ are simple and have the form $\xi_0 + 2\pi n /R$ for certain complex number $\xi_0$ and integer number $n$. Then the roots of the function $\xi^3 - a \xi^2 - b \xi + p$ must also be simple and coincide with the roots of the numerator. As a result, for certain complex number $\xi_0$ and natural $k$, $l$ the roots of the denominator can be written in such a form:   
$$
\xi_0,\quad \xi_1 = \xi_0 + k \frac{2\pi}R, \quad  \xi_2 = (k+l) \frac{2\pi}R.
$$
Exploiting the Vieta formulas
$$
\xi_0 + \xi_1 + \xi_2 = a, \quad \xi_0 \xi_1 + \xi_0 \xi_2 + \xi_1 \xi_2 = -b,
$$
we express $\xi_0$ from the first one, substitute it into the second one and derive an equality
$$
a^2 + 3b = (k^2 +kl + l^2) \frac{4\pi^2}{R^2}.
$$
\end{proof}

\begin{remark}\label{R4}
It can be shown, that the restriction on the size of the interval is also sufficient for existence of such eigenfunctions, but this is not used further.
\end{remark}

\begin{lemma}\label{L4}
For $T>0$ let $\mathcal F_T$ denote the space of initial functions $u_0 \in L_2(I)$ such that  $Pu_0=0$ in $L_2(0,T)$. Then  $\mathcal F_T = \{0\}$ for all $T>0$ if $3b+a^2 \leq 0$ or inequality \eqref{5} is satisfied if $3b + a^2 >0$.
\end{lemma}

\begin{proof}
It is obvious that $\mathcal F_{T'} \subseteq \mathcal F_T$ if $T<T'$.

For any $T>0$ the set $\mathcal F_T$ is a finite-dimensional vector space, In fact, if $u_{0n}$ is a sequence in a unit ball $\{y\in \mathcal F_T: \|y\|_{L_2(I)} \leq 1\}$ it follows from \eqref{12} that the corresponding sequence of weak solutions $\{u_n\}$ is bounded in $L_2(0,T;H^1(I)))$ and, therefore, the set 
\begin{equation}\label{17}
u_{nt} = -u_{nxxx} + i a u_{nxx} - b u_{nx}
\end{equation}
is bounded in $L_2(0,T; H^{-2}(I))$. With the use of the continuous embeddings $H^1(I) \subset
L_2(I) \subset H^{-2}(I)$, where the first one is compact, by the standard argument (see \cite{Lions}) we obtain that the set $u_n$ is relatively compact in $L_2(Q_T)$. Extracting the subsequence, we derive that it is convergent in $L_2(Q_T)$, whence it follows from \eqref{16} that the corresponding subsequence of $u_{0n}$ is convergent in $L_2(I)$. It means that the considered unit ball is compact and the Riesz theorem (see \cite{Yosida}) implies that the space $\mathcal F_T$ has a finite dimension.

Let $T'>0$ is given. To prove  that $\mathcal F_{T'} = \{0\}$, it is sufficient to find $T \in (0,T')$ such that $\mathcal F_{T} = \{0\}$. Since the map $T \mapsto \dim(\mathcal F_T)$ is non-increasing and step-like, there exist $T,\epsilon >0$ such that $T<T+\epsilon<T'$ and $\dim\mathcal F_T = \dim\mathcal F_{T+\epsilon}$. Let $u_0 \in \mathcal F_T$ and $t\in (0,\epsilon)$. Since $e^{tA} e^{\tau A} u_0 = e^{(t+\tau)A} u_0$ for $\tau\geq 0$ and $u_0 \in \mathcal F_{T+\epsilon}$, then
\begin{equation}\label{18}
\frac{e^{tA} u_0 - u_0}t \in \mathcal F_T.
\end{equation}
Let $\mathcal M_T = \{u = e^{\tau A} u_0 : \tau \in [0,T], u_0 \in \mathcal F_T\} \subset C([0,T];L_2(I))$. Since $u \in H^1(0,T+\epsilon; H^{-2}(I))$, there exists
$$
\lim\limits_{t\to +0} \frac{u(\tau +t) - u(\tau)}t = u'(\tau) \quad \text{in} \ L_2(0,T; H^{-2}(I)).
$$
On the other hand, by \eqref{18}
$$
\frac{u(\tau +t) - u(\tau)}t  = e^{\tau A} \frac{e^{tA} u_0 - u_0}t \in \mathcal M_T 
$$
for $t\in (0,\epsilon)$ and $\mathcal M_T$ is closed in $L_2(0,T; H^{-2}(I))$ since $\dim\mathcal M_T <\infty$. Therefore, $u' \in C([0,T];L_2(I))$ and $u \in C^1([0,T]; L_2(I))$. In particular,
$$
u'(0) = \lim\limits_{t\to +0} \frac{e^{tA} u_0 - u_0}t \quad \text{in} \ L_2(I).
$$
Therefore,
$$
u_0 \in D(A), \quad Au_0 = u'(0) \in \mathcal F_T, \quad Pu_0 = u_x\big|_{x=0} \in C[0,T]
$$
(the last property holds since $u \in C([0,T]; H^3(I))$). Hence,
$$
u_0'(0) = u_x(0,0) = 0.
$$
Since $\dim\mathcal F_T <\infty$, if $\mathcal F_T \ne \{0\}$ the map $u_0\in \mathcal F_T \mapsto Au_0 \in \mathcal F_T$ has at least one nontrivial eigenfunction, which contradicts Lemma~\ref{L3}.
\end{proof}

\begin{lemma}\label{L5}
Let either $3b + a^2 \leq 0$ or $3b + a^2 >0$ and inequality \eqref{5} be satisfied. Then for any $T>0$ there exists a constant $c = c(T,R)$ such that for any $u_0 \in L_2(I)$
\begin{equation}\label{19}
\|u_0\|_{L_2(I)} \leq c \|Pu_0 \|_{L_2(0,T)}.
\end{equation}
\end{lemma}

\begin{proof}
We argue by contradiction. If \eqref{19} is not verified there exists a sequence $\{u_{0n}\}_{n\in \mathbb N}$ such that $\|u_{0n}\|_{L_2(I)} =1$ $\forall n$ and $\|Pu_0\|_{L_2(0,T)} \to 0$ when $n\to +\infty$. As in the proof of the previous lemma the corresponding sequence of weak solutions $\{u_n\}$ is bounded in $L_2(0,T;H^1(I))$ and according to \eqref{17} the sequence $u_{nt}$ is bounded in $L_2(0,T;H^{-2}(I))$. Again as in the proof of the previous lemma extract a subsequence of $\{u_n\}$, for simplicity also denoted as  $\{u_n\}$, such that 
it is convergent in $L_2(Q_T)$. Then by \eqref{16} $\{u_{0n}\}$ converges in $L_2(I)$ to certain function $u_0$. Inequality \eqref{15} implies that $Pu_{0n} \to Pu_0$ in $L_2(0,T)$. Then $\|u_0\|_{L_2(I)} =1$ and $\|Pu_0\|_{L_2(0,T)} = 0$, which contradicts Lemma~\ref{L4}.
\end{proof}

Now consider the non-homogeneous linear equation
\begin{equation}\label{20}
i u_t  + a u_{xx} + i b u_x + i u_{xxx}  = f_0(t,x) - f_{1x}(t,x).
\end{equation}
The notion of a weak solution to the corresponding initial-boundary value problem with initial and boundary conditions \eqref{2}, \eqref{3} is similar to Definition~\ref{D1}. In particular, the corresponding integral identity (for the same test functions as in Definition~\ref{D1}) is written as follows:
\begin{multline}\label{21}
\iint_{Q_T} \Bigl[ i u\phi_t  + a u_x \phi_x + i b u \phi_x - i u_x \phi_{xx} + f_0\phi + f_1\phi_x\Bigr]\, dxdt \\ +
i \int_0^R u_0 \phi\big|_{t=0} \,dx + i \int_0^T h \phi_x\big|_{x=R} \,dt =0.
\end{multline}
The following result is established in \cite{FM}.

\begin{lemma}\label{L6}
Let $u_0\in L_2(I)$, $\mu, \nu \in H^{1/3}(0,T)$, $h\in L_2(0,T)$, $f_0 \in L_{1}(0,T;L_2(I))$, $f_1 \in L_2(Q_T)$. Then there exists a unique weak solution $u = S(u_0, \mu, \nu, h, f_0, f_1)\in X(Q_T)$ of problem \eqref{20}, \eqref{2}, \eqref{3} and 
\begin{multline}\label{22}
\|u\|_{X(Q_{T})} \leq c(T)\Bigl[ \|u_0\|_{L_2(I)} + \|\mu\|_{H^{1/3}(0,T)} + \|\nu\|_{H^{1/3}(0,T)}  + \|h\|_{L_2(0,T)} \\+
\|f_0\|_{L_{1}(0,T;L_2(I))} + \|f_1\|_{L_2(Q_{T})}\Bigr],
\end{multline}
for certain constant $c(T)$, non-decreasing with respect to $T$.
\end{lemma}

\begin{remark}\label{R5}
Let
$$
 S_T(u_0, \mu, \nu, h, f_0, f_1) \equiv S(u_0, \mu, \nu, h, f_0, f_1)\big|_{t=T}
$$
Then it follows from \eqref{22} that
\begin{multline}\label{23}
\|S_T(u_0, \mu, \nu, h, f_0, f_1)\|_{L_2(0,R)} \leq c(T)\Bigl[ \|u_0\|_{L_2(I)} + \|\mu\|_{H^{1/3}(0,T)} + \|\nu\|_{H^{1/3}(0,T)} \\ + \|h\|_{L_2(0,T)} +
\|f_0\|_{L_{1}(0,T;L_2(I))} + \|f_1\|_{L_2(Q_{T})}\Bigr].
\end{multline}
Note also that $S(u_0,0,0,0,0,0) = \{e^{tA}u_0 : t\in [0,T]\}$.
\end{remark}

\begin{corollary}\label{C2}
Let the hypothesis of Lemma~\ref{L6} be satisfied, then for $u = S(u_0, \mu, \nu, h, f_0, f_1)$ and any function $\phi\in C^1([0,T]; L_2(I)) \cap C([0,T];(H^{3}\cap H^1_0)(I))$, 
$\phi_x\big|_{x=0} \equiv 0$, the following identity holds:
\begin{multline}\label{24}
\iint_{Q_T} \Bigl[u (i\phi_t  - a \phi_{xx} + i b \phi_x + i \phi_{xxx}) + f_0\phi + f_1\phi_x\Bigr]\, dxdt  +
i \int_0^R u_0 \phi\big|_{t=0} \,dx   \\ 
- i\int_0^R (u \phi)\big|_{t=T} \,dx +  \int_0^T \mu (i \phi_{xx} - a\phi_x)\big|_{x=0} \,dt +  i \int_0^T (h \phi_x - \nu \phi_{xx}) \big|_{x=R} \,dt =0.
\end{multline}
\end{corollary}

\begin{proof}
Let $\eta(x)$ be a cut-off function, namely, $\eta$ is an infinitely smooth non-decreasing function on $\mathbb R$ such that $\eta(x)=0$ for $x\leq 0$, $\eta(x)=1$ for $x\geq 1$, $\eta(x)+\eta(1-x) \equiv 1$. Denote $\phi_\varepsilon(t,x) \equiv \phi(t,x) \eta\bigl((T-t)/\varepsilon\bigr)$, then $\phi_\varepsilon$ satisfies the assumptions on test functions from Definition~\ref{D1}. Write the corresponding equality\eqref{21}:
\begin{multline*}
\iint_{Q_T} \Bigl[ i u\phi_{\varepsilon t}  + a u_x \phi_{\varepsilon x} + i b u \phi_{\varepsilon x} - i u_x \phi_{\varepsilon xx} + f_0\phi_\varepsilon + f_1\phi_{\varepsilon x}\Bigr]\, dxdt \\ + i \int_0^R u_0 \phi_\varepsilon\big|_{t=0} \,dx + i \int_0^T h \phi_{\varepsilon x}\big|_{x=R} \,dt =0.
\end{multline*}
Here 
$$
\phi_{\varepsilon t}(t,x) = \phi_t(t,x) \eta\Bigl(\frac{T-t}\varepsilon\Bigr) -\frac1\varepsilon \phi(t,x)\eta'\Bigl(\frac{T-t}\varepsilon\Bigr).
$$
Since $u\phi \in C([0,T];L_1(I)$, 
$$
-\frac1\varepsilon \iint_{Q_T} u \phi \eta'\Bigl(\frac{T-t}\varepsilon\Bigr) \,dxdt \to -\int_0^R (u \phi)\big|_{t=T} \,dx
$$
when $\varepsilon \to +0$. Therefore, passing to the limit when $\varepsilon \to +0$ and integrating by parts we derive equality \eqref{24}.
\end{proof}

Establish a result on boundary controllability in the linear case.

\begin{theorem}\label{T2}
Let $u_0, u_T \in L_2(I)$, $\mu, \nu \in H^{1/3}(0,T)$, $f_0\in L_1(0,T;L_2(I))$, $f_1 \in L_2(Q_T)$. Assume also that if $3b + a^2 > 0$ condition \eqref{5} is satisfied.
Then there exists a function $h\in L_2(0,T)$ and the corresponding unique solution of problem \eqref{20}, \eqref{2}, \eqref{3} $u\in X(Q_T)$, verifying condition \eqref{4}.
\end{theorem}

\begin{proof}
Assume first that $u_0\equiv 0$, $\mu=\nu \equiv 0$, $f_0 =f_1 \equiv 0$. For $h\in L_2(0,T)$ consider the solution $u = S(0,0,0,h,0,0) \equiv S_0h \in X(Q_T)$ of the corresponding problem \eqref{20}, \eqref{2}, \eqref{3}; let $S_{0T}h \equiv S_{0}h \big|_{t=T}$. Then estimate \eqref{23} implies that $S_{0T}$ is the linear bounded operator from $L_2(0,T)$ to $L_2(I)$.

Consider the backward problem in $Q_T$
\begin{gather}\label{25}
i \phi_t - a \phi_{xx} + i b \phi_x + i \phi_{xxx} =0, \\
\label{26}
\phi\big|_{t=T} = \phi_0(x), \quad \phi\big|_{x=0} = \phi_x\big|_{x=0} = \phi\big|_{x=R} =0.
\end{gather}
Then this problem is equivalent to the problem for the function $\widetilde\phi(t,x) \equiv \phi(T-t, R-x)$
\begin{gather*}
i \widetilde\phi_t + a \widetilde\phi_{xx} + i b \widetilde\phi_x + i \widetilde\phi_{xxx} =0, \\
\widetilde\phi\big|_{t=0} = \widetilde\phi_0(x) \equiv \phi_0(R-x), \quad \widetilde\phi\big|_{x=0}  = \widetilde\phi\big|_{x=R} = \widetilde\phi_x\big|_{x=R} =0.
\end{gather*} 
Let
$$
(\Lambda\phi_0)(t) \equiv -(P_0\widetilde\phi_0)(T-t).
$$
Then it follows from Corollary~\ref{C1} that $\Lambda\phi_0 = \phi_x\big|_{x=R}$ if $\phi_0 \in D(A^*)$ and from inequalities \eqref{15}, \eqref{19} that
\begin{equation}\label{27}
\|\Lambda\phi_0\|_{L_2(0,T)} \leq \|\phi_0\|_{L_2(I)} \leq c \|\Lambda\phi_0\|_{L_2(0,T)}.
\end{equation}

In the case $\phi_0 \in D(A^*)$ the corresponding solution of problem \eqref{25}, \eqref{26} satisfies the assumptions on the functions $\phi$ from Corollary~\ref{C2}. Write equality \eqref{24} for $u = S_0h$ and $\bar\phi$, then
\begin{equation}\label{28}
\int_0^R S_{0T}h  \cdot \overline{\phi_0} \,dx = \int_0^T h \cdot \overline{\Lambda\phi_0} \,dt.
\end{equation}
By continuity this equality can be extended to the case $h\in L_2(0,T)$, $\phi_0 \in L_2(I)$. Let $B \equiv S_{0T} \circ \Lambda$, then according to \eqref{27} and the aforementioned properties of the operator $S_{0T}$ the operator $B$ is bounded in $L_2(I)$. Moreover, \eqref{27} and \eqref{28} provide that 
$$
( B\phi_0, \phi_0)_{L_2(I)} = \int_0^R ( S_{0T} \circ \Lambda)\phi_0 \cdot \overline{\phi_0} \,dx = \int_0^T |\Lambda\phi_0|^2 \,dt \geq \frac1{c^2} \|\phi_0\|^2_{L_2(I)}.
$$
Application of the Lax--Milgram theorem (see, \cite{Yosida}) implies, that the operator $B$ is invertible and $B^{-1}$ is bounded in $L_2(I)$. Let
\begin{equation}\label{29}
\Gamma \equiv \Lambda \circ B^{-1}.
\end{equation}
This operator is bounded from $L_2(I)$ to $L_2(0,T)$. Then $h\equiv \Gamma u_T$ ensures the desired result in the considered case, since
$$
(S_{0T} \circ \Gamma) u_T = (S_{0T} \circ \Lambda \circ B^{-1}) u_T = u_T.
$$

In the general case the desired solution is constructed by formulas 
\begin{equation}\label{30}
h \equiv \Gamma\bigl( u_T - S_T(u_0,\mu,\nu,0,f_0,f_1)\bigr), \quad u \equiv S(u_0,\mu,\nu,0,f_0,f_1) + S_0h.
\end{equation}
\end{proof}

\begin{remark}\label{R6}
Note that the function $h$ can not be defined in a unique way. Indeed, choose $h \ne 0$ in $L_2(0,T/2)$. Move the time origin to the point $T/2$ and for $u_0 \equiv S_{T/2}(0,0,0,h,0,0)$ and $u_T\equiv 0$ construct the solution of the corresponding boundary controllability problem,  which is, of course, nontrivial. However, $h\equiv 0$ and $u\equiv 0$ solve the same problem.
\end{remark}

\section{Nonlinear problem}\label{S3}

Now we pass to the nonlinear equation and first of all establish three auxiliary estimates.

\begin{lemma}\label{L7}
Let $p\in [1,4]$, then for any functions $u, v, \in X(Q_T)$
\begin{equation}\label{31}
\bigl\| |u|^{p} v \bigr\|_{L_{1}(0,T;L_2(I))}
\leq c (T^{(4-p)/4} +T) \|u\|_{X(Q_T)}^{p}  \|v\|_{X(Q_T)}.
\end{equation}
\end{lemma}

\begin{proof}
Applying interpolating inequality \eqref{8} we find that
\begin{equation}\label{32}
\bigl\| |u|^{p} v \bigr\|_{L_2(I)} \leq \|u\|_{L_{\infty}(I)}^{p} \|v\|_{L_{2}(I)}  \\ \leq
c\left(\|u_x\|_{L_2(I)}^{p/2}\|u\|_{L_2(I)}^{p/2} + \|u\|_{L_2(I)}^{p}\right)
\|v\|_{L_2(I)},
\end{equation}
and, applying the H\"{o}lder inequality, we find that 
\begin{multline*}
\bigl\| \|u_x\|_{L_2(I)}^{p/2} \|v\|_{L_2(I)}\bigr\|_{L_{1}(0,T)}  \\ \leq
T^{(4-p)/4} \sup\limits_{t\in [0,T]}  
\Bigl[\|u(t,\cdot)\|_{L_2(I)}^{p/2} \|v(t,\cdot)\|_{L_2(I)}\Bigr]
\|u_x\|_{L_2(Q_T)}^{p/2} \\  \leq
T^{(4-p)/4} \|u\|_{X(Q_T)}^{p} \|v\|_{X(Q_T)}.
\end{multline*}
Finally,
\begin{multline*}
\bigl\| \|u\|_{L_2(I)}^{p} \|v\|_{L_2(I)} \bigr\|_{L_{1}(0,T)}  \leq
T \sup\limits_{t\in [0,T]} \Bigl[ \|u(t,\cdot)\|_{L_2(I)}^{p}  \|v(t,\cdot)\|_{L_2(I)}\Bigr]    \\ \leq T\|u\|_{X(Q_T)}^{p} \|v\|_{X(Q_T)}.
\end{multline*}
\end{proof}

\begin{lemma}\label{L8}
Let $p\in [1,2]$, then for any functions $u, v \in X(Q_T)$
\begin{equation}\label{33}
\bigl\| |u|^{p} v \bigr\|_{L_{2}(Q_T)}
\leq c (T^{(2-p)/4} +T^{1/2}) \|u\|_{X(Q_T)}^{p}  \|v\|_{X(Q_T)}.
\end{equation}
\end{lemma}

\begin{proof}
Applying estimate \eqref{32} and the H\"{o}lder inequality,
we find that 
\begin{multline*}
\bigl\| \|u_x\|_{L_2(I)}^{p/2} \|v\|_{L_2(I)}\bigr\|_{L_{2}(0,T)}  \\ \leq
T^{(2-p)/4} \sup\limits_{t\in [0,T]} 
\Bigl[\|u(t,\cdot)\|_{L_2(I)}^{p/2} \|v(t,\cdot)\|_{L_2(I)} \Bigr] 
 \|u_x\|_{L_2(Q_T)}^{p/2}\\ \leq
T^{(2-p)/4} \|u\|_{X(Q_T)}^{p} \|v\|_{X(Q_T)}.
\end{multline*}
Finally,
\begin{multline*}
\bigl\| \|u\|_{L_2(I)}^{p} \|v\|_{L_2(I)} \bigr\|_{L_{2}(0,T)} \leq
T^{1/2}\sup\limits_{t\in [0,T]} \Bigl[ \|u(t,\cdot)\|_{L_2(I)}^{p}  \|v(t,\cdot)\|_{L_2(I)} \Bigr]\\ \leq
T^{1/2}\|u\|_{X(Q_T)}^{p} \|v\|_{X(Q_T)}.
\end{multline*}
\end{proof}

\begin{lemma}\label{L9}
Let $p\in [1,2]$, then for any functions $u, v, w \in X(Q_T)$
\begin{equation}\label{34}
\bigl\| |u|^{p-1} v w_x\bigr\|_{L_{1}(0,T;L_2(I))}
\leq c (T^{(2-p)/4} +T^{1/2}) \|u\|_{X(Q_T)}^{p-1}  \|v\|_{X(Q_T)}  \|w\|_{X(Q_T)}.
\end{equation}
\end{lemma}

\begin{proof}
Applying interpolating inequality \eqref{8} we find that
\begin{multline*}
\bigl\| |u|^{p-1} v w_x\bigr\|_{L_2(I)} 
\leq \|u\|_{L_{\infty}(I)}^{p-1} \|v\|_{L_{\infty}(I)}  \| w_x\|_{L_2(I)} \\ \leq
c\left(\|u_x\|_{L_2(I)}^{(p-1)/2}\|u\|_{L_2(I)}^{(p-1)/2} + \!\|u\|_{L_2(I)}^{p-1}\right) \!\!
\left(\|v_x\|_{L_2(I)}^{1/2}\|v\|_{L_2(I)}^{1/2} + \!\|v\|_{L_2(I)}\right) \!
\|w_x\|_{L_2(I)}.
\end{multline*}
Here because of the restriction on $p$
$$
1- \frac{p-1}{4} -\frac14 - \frac{1}2 = \frac {2-p}4 \geq 0
$$
and, applying the H\"{o}lder inequality, we find that 
\begin{multline*}
\bigl\| \|u_x\|_{L_2(I)}^{(p-1)/2} \|v_x\|_{L_2(I)}^{1/2} \|w_x\|_{L_2(I)}\|u\|_{L_2(I)}^{(p-1)/2} \|v\|_{L_2(I)}^{1/2} \bigr\|_{L_{1}(0,T)}  \\ \leq
T^{(2-p)/4} \sup\limits_{t\in [0,T]} 
\Bigl[\|u(t,\cdot)\|_{L_2(I)}^{(p-1)/2} \|v(t,\cdot)\|_{L_2(I)}^{1/2} \Bigr] 
 \|u_x\|_{L_2(Q_T)}^{(p-1)/2} \|v_x\|_{L_2(Q_T)}^{1/2} \|w_x\|_{L_2(Q_T)} \\ \leq
T^{(2-p)/4} \|u\|_{X(Q_T)}^{p-1} \|v\|_{X(Q_T)} \|w\|_{X(Q_T)}.
\end{multline*}
Finally,
\begin{multline*}
\bigl\| \|u\|_{L_2(I)}^{p-1} \|v\|_{L_2(I)} \|w_x\|_{L_2(I)}\bigr\|_{L_{1}(0,T)} \\ \leq
T^{1/2} \sup\limits_{t\in [0,T]} \Bigl[ \|u(t,\cdot)\|_{L_2(I)}^{p-1}  \|v(t,\cdot)\|_{L_2(I)} \Bigr]  \|w_x\|_{L_2(Q_T)}  \\ \leq
T^{1/2}\|u\|_{X(Q_T)}^{p-1} \|v\|_{X(Q_T)} \|w\|_{X(Q_T)}.
\end{multline*}
\end{proof}

\begin{proof}[Proof of existence part of Theorem \ref{T1}]
For a function $v\in X(Q_T)$ set
\begin{gather}\label{35}
f_{00}(t,x;v) \equiv  f(t,x) - \lambda |v|^{p_0} v, \quad 
f_{01}(t,x;v) \equiv i \gamma |v|^{p_1} v_x, \\
\label{36}
f_0(t,x;v) \equiv f_{00}(t,x;v) + f_{01}(t,x;v), \quad
f_1(t,x;v) \equiv i (\beta + \gamma) |v|^{p_1} v
\end{gather}
and consider the corresponding controllability problem for equation \eqref{20}. Lemmas \ref{L7}--\ref{L9} provide that $f_0 \in L_1(0,T;L_2(I))$, $f_1 \in L_2(Q_T)$. Then Theorem \ref{T2} implies that there exist a function $h\in L_2(0,T)$ and the corresponding unique solution $u\in X(Q_T)$ of problem \eqref{20}, \eqref{2}, \eqref{3}, verifying condition \eqref{4}. Therefore, on the space $X(Q_T)$ one can define a map $\Theta$, where $u = \Theta v$ is given by formulas \eqref{30}.
Moreover, according to \eqref{31} 
\begin{equation}\label{37}
\|f_{00}(\cdot,\cdot;v)\|_{L_1(0,T;L_2(I))} \leq  \|f\|_{L_1(0,T;L_2(I))} + c(T) \|v\|_{X(Q_T)}^{p_0+1}
\end{equation}
and according to \eqref{33}, \eqref{34}
\begin{equation}\label{38}
\|f_{01}(\cdot,\cdot,v)\|_{L_1(0,T;L_2(I))},  \|f_1(\cdot,\cdot,v)\|_{L_2(Q_T)} 
\leq c(T)\|v\|_{X(Q_T)}^{p_1+1}.
\end{equation}

Apply Lemma~\ref{L6}, then inequality \eqref{22} and formulas \eqref{30}
imply that
\begin{equation}\label{39}
\|\Theta v\|_{X(Q_T)} \leq c(T)c_0 + c(T) \left(\|v\|_{X(Q_T)}^{p_0+1} 
+\|v\|_{X(Q_T)}^{p_1+1}\right),
\end{equation}
where the value of $c_0$ is given by \eqref{7}.

Next, for any functions $v_1, v_2 \in X(Q_T)$
\begin{equation}\label{40}
|f_{00}(t,x;v_1) - f_{00}(t,x;v_2)| \leq c\bigl(|v_1|^{p_0} + |v_2|^{p_0}\bigr) |v_1 -v_2|,
\end{equation}
\begin{multline}\label{41}
|f_{01}(t,x;v_1) - f_{01}(t,x;v_2)| \leq c\bigl(|v_1|^{p_1} + |v_2|^{p_1}\bigr)
 |v_{1x} - v_{2x}| \\+
c\bigl(|v_1|^{p_1-1} + |v_2|^{p_1-1} \bigr) \bigl(|v_{1x}| + |v_{2x}|\bigr) |v_1-v_2|,  
\end{multline}
\begin{equation}\label{42}
|f_1(t,x;v_1) - f_1(t,x;v_2)| \leq c\bigl(|v_1|^{p_1} + |v_2|^{p_1}\bigr) |v_1-v_2|,
\end{equation}
therefore, similarly to \eqref{37}
\begin{multline*}
\|f_{00}(\cdot,\cdot;v_1) - f_{00}(\cdot,\cdot;v_2)\|_{L_1(0,T;L_2(I))} \\ \leq
c(T) \bigl(\|v_1\|_{X(Q_T)}^{p_0} + \|v_2\|_{X(Q_T)}^{p_0}\bigr) \|v_1-v_2\|_{X(Q_T)},
\end{multline*}
and similarly to \eqref{38}
\begin{multline*}
\|f_{01}(\cdot,\cdot;v_1) - f_{01}(\cdot,\cdot;v_2)\|_{L_1(0,T;L_2(I))},  \|f_1(\cdot,\cdot,v_1) - f_1(\cdot,\cdot;v_2)\|_{L_2(Q_T)} \\ \leq 
c(T) \bigl(\|v_1\|_{X(Q_T)}^{p_1} + \|v_2\|_{X(Q_T)}^{p_1}\bigr) \|v_1 - v_2\|_{X(Q_T)}.
\end{multline*}

Since
\begin{multline*}
\Theta v_1 - \Theta v_2 =  S\bigl(0,0,0,0, f_0(t,x;v_1) - f_0(t,x;v_2), f_1(t,x;v_1) - f_1(t,x;v_2)\bigr) \\-
(S_0\circ \Gamma)\Bigl(S\bigl(0,0,0,0, f_0(t,x;v_1) - f_0(t,x;v_2), f_1(t,x;v_1) - f_1(t,x;v_2)\bigr)\Bigr),
\end{multline*}
it follows similarly to \eqref{39} that
\begin{multline}\label{43}
\|\Theta v_1 - \Theta v_2\|_{X(Q_T)}  \\ \leq 
c(T)\left( \|v_1\|_{X(Q_T)}^{p_0} + \|v_1\|_{X(Q_T)}^{p_1} + 
\|v_2\|_{X(Q_T)}^{p_0} + \|v_2\|_{X(Q_T)}^{p_0} \right)\|v_1-v_2\|_{X(Q_T)}.
\end{multline}

Now choose $r>0$ such that
$$
r^{p_0} + r^{p_1} \leq \frac1{4c(T)}
$$
and then $\delta>0$ such that
$$
\delta \leq \frac{r}{2c(T)}.
$$
Then it follows from \eqref{39} and \eqref{43} that on the ball $\overline X_r(Q_T)$ the map $\Theta$ is a contraction. Its unique fixed point $u\in X(Q_T)$ is the desired solution. 
\end{proof}

The contraction principle used in the previous proof ensures uniqueness of the solution $u$ only in the ball $\overline X_r(Q_T)$. The next theorem provides uniqueness in the whole space $X(Q_T)$, which finishes the proof of Theorem~\ref{T1}.

\begin{theorem}\label{T3}
A weak solution of problem \eqref{1}--\eqref{3} is unique in the space $X(Q_T)$ if $p_0\in [1,4]$, $p_1\in [1,2]$.
\end{theorem}

\begin{proof}
Let $u, \widetilde u \in X(Q_T)$ be two weak solutions of the same problem \eqref{1}--\eqref{3}. Denote $w \equiv u - \widetilde u$, then the function $w\in X(Q_T)$ is the weak solution of the problem of \eqref{9}, \eqref{10} type for $f \equiv f_0 - f_{1x}$, where 
\begin{gather*}
f_0(t,x) \equiv f_{00}(t,x;u) - f_{00}(t,x;\widetilde u) + f_{01}(t,x;u) - f_{01}(t,x;\widetilde u), \\
f_1(t,x) \equiv f_1(t,x;u) - f_1(t,x;\widetilde u),
\end{gather*}
given by formulas \eqref{35}, \eqref{36}. Similarly to the previous proof $f_0 \in L_1(0,T;L_2(I))$, $f_1 \in L_2(Q_T)$. Then the corresponding equality \eqref{13} in the case $\rho(x) \equiv 1+x$ yields that
\begin{multline}\label{44}
\frac{d}{dt}\int_0^R (1+x)|w(t,x)|^2\,dx + 3\int_0^R |w_x|^2 \,dx  = b \int_0^R |w|^2 \,dx   
+2a \Im \int_0^R w_x \bar{w} \,dx \\
+ 2\Im \int_0^R (1+x)f_0\bar{w} \,dx 
+ 2 \Im \int_0^R f_1 \bigl((1+x) \bar{w}_x + \bar{w}\bigr) \,dx.
\end{multline}
To estimate the last two integrals in the right-hand side of \eqref{44} apply inequalities \eqref{40}--\eqref{42}. Then
$$
\Bigl|\int_0^R (1+x)\bigl(f_{00}(t,x;u) - f_{00}(t,x;\widetilde u)\bigr) \bar{w} \,dx  \Bigr| \leq
c \esssup\limits_{x\in I} \bigl(|u|^{p_0} + |\widetilde u|^{p_0}\bigr) \int_0^R |w|^2\, dx,
$$
\begin{multline*}
\Bigl|\int_0^R (1+x)\bigl(f_{1}(t,x;u) - f_{1}(t,x;\widetilde u)\bigr) \bar{w_x} \,dx \Bigr| \leq
c \esssup\limits_{x\in I} \bigl(|u|^{p_1} + |\widetilde u|^{p_1}\bigr) \int_0^R |w w_x|\, dx \\
\leq \varepsilon \int_0^R |w_x|^2 \,dx  + c(\varepsilon)
\esssup\limits_{x\in I} \bigl(|u|^{2p_1} + |\widetilde u|^{2p_1}\bigr) \int_0^R |w|^2 \,dx,
\end{multline*}
where $\varepsilon>0$ can be chosen arbitrarily small,
\begin{multline*}
\Bigl|\int_0^R (1+x)\bigl(f_{01}(t,x;u) - f_{01}(t,x;\widetilde u)\bigr) \bar{w} \,dx \Bigr|  \leq
c \esssup\limits_{x\in I} \bigl(|u|^{p_1} + |\widetilde u|^{p_1}\bigr) \int_0^R |w w_x|\, dx
\\ +c \esssup\limits_{x\in I} \Bigl[\bigl(|u|^{p_1-1} + |\widetilde u|^{p_1-1}\bigr) |w| \bigr]
\int_0^R \bigl(|u_x| + |\widetilde u_x|\bigr) |w| \,dx,
\end{multline*}
where the first term in the right-hand side is already estimated above, while the second one does not exceed
\begin{multline*}
c \Bigl(\int_0^R |w_x|^2 \,dx \Bigr)^{1/4}  
\esssup\limits_{x\in I} \bigl(|u|^{p_1-1} + |\widetilde u|^{p_1-1}\bigr) 
\Bigl(\int_0^R \bigl(|u_x|^2 + |\widetilde u_x|^2\bigr) \,dx \Bigr)^{1/2} \\ \times
\Bigl(\int_0^R |w|^2 \,dx \Bigr)^{3/4} \leq
\varepsilon \int_0^R |w_x|^2 \,dx  + c(\varepsilon) 
\Bigl[\esssup\limits_{x\in I} \bigl(|u|^{4(p_1-1)} + |\widetilde u|^{4(p_1-1)}\bigr)  \\ +
\int_0^R \bigl(|u_x|^2 + |\widetilde u_x|^2\bigr) \,dx \Bigr] 
\int_0^R |w|^2 \,dx
\end{multline*}
(here estimate \eqref{8} is used in the case of the space $H^1_0(I)$).
Note that according to \eqref{8} $u, \widetilde u \in L_4(0,T;L_\infty(I))$. Then since $p_0, 2p_1, 4(p_1-1) \leq 4$ it follows from \eqref{44} that
$$
\frac{d}{dt}\int_0^R (1+x)|w(t,x)|^2\,dx \leq \omega(t) \int_0^R (1+x)|w(t,x)|^2\,dx,
$$
for certain function $\omega\in L_1(0,T)$. Application of the Gronwall lemma yields that $w\equiv 0$.

\end{proof}

\end{document}